\documentclass[10pt]{amsart}
\usepackage{amsmath,amssymb}
\usepackage[latin1]{inputenc}
\usepackage[dvips]{graphicx}

\newtheorem{theorem}{Theorem}[section]
\newtheorem{lemma}[theorem]{Lemma}
\newtheorem{proposition}[theorem]{Proposition}

\newtheorem{corollary}[theorem]{Corollary}
\newtheorem{remark}[theorem]{Remark}

\newtheorem{problem}[theorem]{Problem}

\newcommand{\Gal}{{\rm Gal}}

\newcommand{\Cen}{\mbox{\rm C}}

\newcommand{\Aut}{\mbox{\rm Aut}}

\newcommand{\inv}{^{-1}}

\newcommand{\Z}{{\mathbb Z}}
\newcommand{\Q}{{\mathbb Q}}

\newcommand{\C}{{\mathbb C}}

\newcommand{\K}{{\mathbb K}}

\newcommand{\GL}{{\rm GL}}

\newcommand{\SGL}{{\rm SGL}}

\newcommand{\tr}{{\rm tr}}

\newcommand{\ind}{{\rm ind}}

\newcommand{\matriz}[1]{\begin{array} #1 \end{array}}
\newcommand{\pmatriz}[1]{\left(\begin{array} #1 \end{array}\right)}
\newcommand{\GEN}[1]{\langle #1 \rangle}

\newcommand{\diag}{\operatorname{diag}}

\newcommand{\Hom}{\operatorname{Hom}}

\title{Zassenhaus conjecture for cyclic-by-abelian groups}

\author{Mauricio Caicedo}
\address{Departamento de Matem\'{a}ticas, Universidad de Murcia,  30100 Murcia, Spain}
\email{mauriciojc02@hotmail.com}

\author{Leo Margolis}
\address{Fachbereich Mathematik, Universitaet Stuttgart, Pfaffenwaldring 57, 70569 Stuttgart, Germany}
\email{leo.imsueden@yahoo.com}

\author{\'{A}ngel del R\'{\i}o}
\address{Departamento de Matem\'{a}ticas, Universidad de Murcia,  30100 Murcia, Spain}
\email{adelrio@um.es}

\thanks{The research is partially supported by the Ministerio de Ciencia y Tecnolog\'{\i}a of Spain and Fundaci\'{o}n S\'{e}neca of Murcia and Die Studienstiftung des deutschen Volkes}

\begin{document}



\begin{abstract}
Zassenhaus Conjecture for torsion units states that every augmentation one torsion unit of the integral group ring of a finite group $G$ is conjugate  to an element of $G$ in the units of rational group algebra $\Q G$. This conjecture has been proved for nilpotent groups, metacyclic groups and some other  families of groups. We prove the conjecture for cyclic-by-abelian groups.
\end{abstract}

\maketitle


In this paper $G$ is a finite group and $RG$ denotes the group ring of $G$ with coefficients in a ring $R$. The units of $RG$ of augmentation one are usually called normalized units. In the 1960s Hans Zassenhaus established a series of conjectures about the finite subgroups of normalized units of $\Z G$. Namely he conjectured that every finite group of normalized units of $\Z G$ is conjugate to a subgroup of $G$ in the units of $\Q G$. These conjecture is usually denoted (ZC3), while the version of (ZC3) for the particular case of subgroups of normalized units with the same cardinality as $G$ is usually denoted (ZC2). These conjectures have important consequences. For example, a positive solution of (ZC2) implies a positive solution for the Isomorphism and Automorphism Problems (see \cite{SehgalBook2} for details). The most celebrated positive result for Zassenhaus Conjectures is due to Weiss \cite{WeissCrelle} who proved (ZC3) for nilpotent groups. However Roggenkamp and Scott founded a counterexample to the Automorphism Problem, and henceforth to (ZC2) (see \cite{Roggenkamp} and \cite{Klingler}). Later Hertweck \cite{HertweckIso} provided a counterexample to the Isomorphism Problem.

The only conjecture of Zassenhaus that is still up is the version for cyclic subgroups namely:

\begin{quote}
\textbf{Zassenhaus Conjecture for Torsion Units (ZC1)}.
If $G$ is a finite group then every normalized torsion unit of $\Z G$ is conjugate in $\Q G$ to an element of $G$.
\end{quote}

Besides the family of nilpotent groups, (ZC1) has been proved for some concrete groups \cite{BovdiHertweck,BovdiHofertKimmerle,HofertKimmerle,LutharPassiA5,LutharTrammaComm,HertweckA6}, for groups having a Sylow subgroup with an abelian complement \cite{HertweckAlgColloq}, for some families of cyclic-by-abelian groups \cite{LutharBhandari,LutharTramaJIndian,LutharSehgal,MarciniakRitterSehgalWeiss,PolcinoSehgalJNumberT,PolcinoRitterSehgal,delRioSehgal,RitterSehgalMathAnn} and some classes of metabelian groups not necessarily cyclic-by-abelian \cite{MarciniakRitterSehgalWeiss,SehgalWeissJAlgebra}.
Other results on Zassenhaus Conjectures can be found in \cite{SehgalBook2,SehgalEncl} and \cite[Section~8]{SehgalHandbook}

The latest and most general result for (ZC1) on the class of cyclic-by-abelian groups is due to Hertweck \cite{HertweckEdinb} who proved (ZC1) for finite groups of the form $G=AX$ with $A$ a cyclic normal subgroup of $G$ and $X$ an abelian subgroup of $G$. This includes the class of metacyclic groups that was not covered in previous results.
The aim of this paper is to prove (ZC1) for arbitrary cyclic-by-abelian groups.
Formally we prove

\begin{quote}
{\bf Theorem}. Let $G$ be a finite cyclic-by-abelian group. Then every normalized torsion unit of $\Z G$ is conjugate in $\Q G$ to an element of $G$.
\end{quote}

Our strategy uses induction on the order of the group $G$ and on the order of the torsion unit. In other words we consider a finite cyclic-by-abelian group $G$, which is a minimal counterexample to (ZC1), and $u$ a torsion unit in $\Z G$, which is a minimal counterexample to (ZC1). Here minimal means ``of minimal order''. In particular, we assume that (ZC1) holds for proper subgroups and quotients of $G$ and for units in proper subgroups of the group generated by $u$.

Most of the ideas used in this paper are either due to or inspired from the techniques introduced by Hertweck in \cite{HertweckEdinb} which we have adapted in some steps of the proof to avoid some difficulties appearing in the general case which are not encountered in the hypothesis of \cite{HertweckEdinb}. For example, the strategy in \cite{HertweckEdinb} for the case when $G=AX$ with $A$ cyclic normal and $X$ abelian, consists in first proving (ZC1) for torsion units $u$ with augmentation 1 modulo $A$, with the help of a result of Cliff and Weiss for the matrix version of Zassenhaus Conjecture (Theorem~\ref{TCW}) and a beautiful use of Weiss permutation module theorem \cite{WeissAnnals}, then using this to prove the results for units with augmentation 1 modulo $C_G(A)$ and then reducing the general case to this special case. In several steps of the proof one uses a faithful linear representation of $A$ and the fact that $C_G(A)=AC_X(A)=AZ(G)$ (and hence $C_G(A)$ is cyclic-by-central). In the words of Hertweck the last fact is ``the main reason for assuming that $A$ is covered by an abelian subgroup--rather than assuming that $G/A$ is abelian''.

In our strategy the subgroup $D=Z(C_G(A))$, for $A$ a cyclic subgroup of $G$ containing $G'$, plays a very important role.
We first prove (ZC1) for units with augmentation 1 modulo $D$ using local methods over the $p$-adic integers and then we prove (ZC1) for the remaining units using the so called Luthar-Passi Method. As $D$ is not cyclic-by-central it is not possible to use neither Cliff-Weiss Theorem, nor a faithful linear character of an appropriate cyclic subgroup of $G$. Instead, for the first part of the proof we adapt the $p$-adic methods of Hertweck and Cliff-Weiss to our situation with a careful revision of their proofs, using linear characters of $D$ with kernels not intersecting $A$, and in the second part we use a family of linear characters of $D$ with kernel not containing any normal subgroup of $G$. This introduces some difficulties in the arguments which makes the proofs more involved than in \cite{HertweckEdinb}.

%
%
%

\section{Notation, preliminaries and some tools}

In this section we establish the general notation and collect some known results, which will be used throughout the paper.

As it is costumary $\varphi$ denotes Euler's totient function. The cardinality of a set $X$ is denoted $\lvert X \rvert$. For every integer $n$ we let $\zeta_n$ denote a fixed complex primitive root of unity of order $n$. The ring of $p$-adic integers, for $p$ a prime integer, is denoted $\Z_p$.
%

We use the standard group theoretical notation. In particular, if $G$ is a group, then $Z(G)$ denotes the center of $G$, $G'$ the commutator subgroup of $G$ and $\exp(G)$ the exponent of $G$. If $g,h\in G$ then $|g|$ denotes the order of $g$, $g^h=h\inv g h$, $(g,h)=g\inv g^h = g\inv h\inv gh$ and $g^G$ denotes the conjugacy class of $g$ in $G$. If $X\subseteq G$ then $\GEN{X}$ denotes the subgroup generated by $X$, $C_G(X)=\{g\in G:(x,g)=1 \ \text{for every } x\in X\}$, the centralizer of $X$ in $G$ and $N_G(X)=\{g\in G:X^g\subseteq X \}$, the normalizer of $X$ in $G$.
Let $p$ be a prime integer.
If $g$ has finite order then $g_p$ and $g_{p'}$ denote the $p$-part and $p'$-part of $g$, respectively.
If $G$ has a unique $p$-Sylow subgroup (respectively a unique $p'$-Hall subgroup) then it is denoted $G_p$ (respectively $G_{p'}$).

In the remainder of this section $R$ stands for a commutative ring.
If $N$ is a normal subgroup of $G$ then the $N$-augmentation map of $RG$ is the unique ring homomorphism $\omega_N:RG \rightarrow R(G/N)$ extending the natural map $G\rightarrow G/N$ and acting on $R$ as the identity. In particular $\omega=\omega_G$ is the augmentation map of $RG$.
Let $r=\sum_{g\in G} r_g g\in RG$ with $r_g\in R$ for every $g$. For every $g\in G$, let $\varepsilon_g^G(r)$ denote the partial augmentation of $r$ in the conjugacy class of $g$ in $G$, that is
    $$\varepsilon^G_g(r) = \sum_{h \in g^G} r_h.$$
If the group $G$ is clear from the context we simply write $\varepsilon_g(r)$.
Conjugacy classes in $RG$ and partial augmentation are strongly related. We collect in the following remark some easy facts about this relation.

\begin{remark}\label{ConjugatePartialAugmentation}
{\rm If $x\in G$ then $\varepsilon_x^G:RG \rightarrow R$ is an $R$-linear map which satisfies $\varepsilon_x^G(uv)=\varepsilon_x^G(vu)$ for every $u,v\in RG$. Using this it is easy to prove that if $u\in RG$ and $g\in G$ are conjugate in $RG$, then $\varepsilon_g^G(u)=1$ and $\varepsilon_x^G(u)=0$ for every $x\in G$ with $x\not\in g^G$. Hence, for such $u$ and $g$ and a normal subgroup $N$ of $G$ we have $\omega_N(u)\ne 0$ if and only if $g\in N$ and in that case $\omega_N(u)=1$.}
\end{remark}

One of the main tools to study (ZC1) is the following well known result which is somehow a converse of Remark~\ref{ConjugatePartialAugmentation} (see e.g. \cite[Lemma~41.5]{SehgalBook2}).

\begin{proposition}\label{ZC1PAugmentation}
Let $u$ be a normalized torsion unit of $\Z G$. Then $u$ is conjugate in $\Q G$ to an element of $G$ if and only if $\varepsilon_g^G(v)\ge 0$ for every $v\in \GEN{u}$ and every $g\in G$.
\end{proposition}

Proposition~\ref{ZC1PAugmentation} is commonly presented in the following equivalent form: A normalized torsion unit $u$ of $\Z G$ is conjugate in $\Q G$ to an element of $G$ if and only if for every $v\in \GEN{u}$, there is $g\in G$ such that for every $x\in G$ we have $\varepsilon_x^G(v)\ne 0$ if and only if $x\in g^G$.

The following proposition collects some results from \cite{HertweckAlgColloq} and \cite{HertweckEdinb} which will be very useful in our arguments.

\begin{proposition}\label{Hertweck}
Let $G$ be a finite group and $p$ a prime integer.
\begin{enumerate}
\item Let $R$ be a $p$-adic ring with quotient field $K$ and $u$ a normalized torsion unit of $RG$.
\begin{enumerate}
\item\label{pAugmentation}  Suppose $\omega_P(u)=1$ for $P$ a normal $p$-subgroup of $G$. Then $u$ is conjugate in $KG$ to an element of $P$.
\item\label{pPartesNoConjugadas} Suppose that the $p$-part of $u$ is conjugate to an element $x$ of $G$ in the units of $RG$ and $g$ is an element of $G$ such that the $p$-parts of $x$ and $g$ are not conjugate in $G$. Then $\varepsilon_g(u)=0$.
\end{enumerate}
\item Let $u$ be a torsion unit of $\Z G$.
\begin{enumerate}
\item\label{AugmentationOrder} If $\varepsilon_g(u)\ne 0$ with $g\in G$ then the order of $g$ divides the order of $u$.
\item\label{ConjugadopAdico} Assume that $\omega_P(u)=1$ with $P$ a cyclic normal $p$-subgroup of $G$. Then $u$ is conjugate in $\Q G$ to an element $x\in P$. If moreover, $\Cen_G(x)$ has a normal $p$-complement, then $u$ and $x$ are conjugate in $\Z_p G$.
\end{enumerate}
\end{enumerate}
\end{proposition}

We now present a matrix version of (ZC1) which was introduced in \cite{MarciniakRitterSehgalWeiss} as a strategy to prove (ZC1) in some cases.
If $k$ is a positive integer then the action of $\omega$ on the matrix entries defines a ring homomorphism $M_k(RG)\rightarrow M_k(R)$. It restricts to a group homomorphism $\GL_k(RG)\rightarrow \GL_k(R)$. Following \cite{CliffWeiss} we let $\SGL_k(RG)$ denote the kernel of this group homomorphism. The matrix version of (ZC1) is the following problem:

\begin{problem}\label{ProblemaCW}
Let $G$ be a finite group. Is every element of finite order of $\SGL_k(\Z G)$ conjugate in $\GL_k(\Q G)$ to a diagonal matrix with diagonal entries in $G$?
\end{problem}

Cliff and Weiss solved this problem for nilpotent groups and arbitrary $k$.

\begin{theorem}\label{TCW} \cite{CliffWeiss}
If $G$ is a finite nilpotent group then Problem~\ref{ProblemaCW} has a positive solution for every $k\ge 1$ if and only if $G$ has at most one non-cyclic Sylow subgroup.
\end{theorem}

Let $N$ be a normal subgroup of $G$ with $k=[G:N]$. Then $RG$ is a $(RG,RN)$-bimodule and the right $RN$-module $RG$ is free with a basis formed by a transversal of $N$ in $G$.
For every $r\in RG$ let $\rho_N(r)$ denote the matrix representation of left multiplication by $r$ with respect to this basis.
This defines a ring homomorphism $\rho_N:RG \rightarrow M_k(RN)$.
If $r$ is a unit of $RG$ such that $\omega_N(r)=1$ then $\rho_N(r)\in \SGL_k(RN)$.
If $x\in N$ then, by \cite[Lemma~41.10]{SehgalBook2}, we have
    \begin{equation}\label{EpsilonTraza}
    \varepsilon_x^N(\tr(\rho_N(r))) = [C_G(x):C_N(x)] \varepsilon_x^G(r),
    \end{equation}
where $\tr(U)$ stands for the trace of a matrix $U\in M_k(RN)$. If $N$ is not commutative and $U,V\in M_k(RN)$ then $\tr(UV)$ and $\tr(VU)$ might be different. However, $\tr(UV)-\tr(VU)\in [RN,RN]$, where $[RN,RN]$ is the $R$ linear span of the Lie products $[a,b]=ab-ba$, with $a,b\in N$. Hence $\varepsilon_x^N(\tr(UV))=\varepsilon_x^N(\tr(VU))$ for every $x\in N$. In particular, if $U$ and $V$ are conjugate in $M_k(RN)$, then $\varepsilon_x^N(\tr(U))=\varepsilon_x^N(\tr(V))$.

A useful technique to deal with Zassenhaus conjecture (in the most general form) is the so called double action formalism introduced by Weiss \cite{WeissAnnals}. Let $G$ and $H$ be groups, $R$ a commutative ring and let $\alpha:H\rightarrow \GL_k(RG)$ be a group homomorphism. Then $M^{\alpha}$ denotes the right $R(H\times G)$ module $(RG)^k$ with multiplication defined by $x\cdot r(h,g)=\alpha(h)\inv x rg$ for $x\in (RG)^k$, $g\in G$, $h\in H$ and $r\in R$. It is easy to see that if $\beta:H\rightarrow \GL_k(RG)$ is another group homomorphism then $M^{\alpha}\cong M^{\beta}$ if and only if $\alpha$ and $\beta$ are conjugate in $\GL_k(RG)$, i.e. if and only if there exist $u\in \GL_k(RG)$ such that $\beta(h)=\alpha(h)^u$ for every $h\in H$. For example, if $C_m=\GEN{c}$, the cyclic group of order $m$ generated by $c$, and $u$ and $v$ are torsion unit of order $m$ in $\GL_k(RG)$ then $u$ and $v$ are conjugate in $\GL_k(RG)$ if and only if the $R(C_m\times G)$ modules $M^{\alpha_u}$ and $M^{\alpha_v}$ are isomorphic, where $\alpha_u$ and $\alpha_v$ are the homomorphisms $C_m\rightarrow \GL_k(RG)$ mapping $c$ to $u$ and $v$ respectively.

Let $m$ be a positive integer, set $\Gamma=C_m\times G$ and let $G[m]$ denote a set of representatives of $G$-conjugacy classes of elements $g\in G$ with $g^m=1$. For every $g\in G$ let $[g]=\GEN{(c,g)}\le \Gamma$. For every prime integer $p$ and $g\in G$ let $G_{g,p}[m]=\{h \in G[m] : g_p \text{ and } h_p \text{ are conjugate in } G\}$. For $H$ a subgroup of a group $K$ we let $\ind_H^K 1$ denote the character of $K$ induced from the trivial character of $H$.

\begin{lemma}\label{character}
Let $G$ be a finite nilpotent group, $u\in \SGL_k(\Z G)$ with $u^m=1$ and let $\chi$ denote the character of $M^{\alpha_u}$. Then
\begin{enumerate}
\item For every $g\in G[m]$ we have $\chi(c,g)\in \lvert C_G(g) \rvert \Z$ and $\chi = \sum_{g\in G[m]} \frac{\chi(c,g)}{\lvert C_G(g) \rvert} \ind_{[g]}^{\Gamma} 1$.
\item For every prime $p$ and every $g\in G[m]$, $\sum_{h \in G_{g,p}[m]} \frac{\chi(c,h)}{\lvert C_G(h)\rvert} \ind_{[h_{p'}]}^{\Gamma_{p'}} 1$ is a proper character of $\Gamma_{p'}$.
\end{enumerate}
\end{lemma}

\begin{proof}
Cliff and Weiss \cite{CliffWeiss} prove that $\chi= \sum_{g\in G[m]} a_g \ind_{[g]}^{\Gamma} 1$ for unique integers $a_g$ and $\sum_{h \in G_{g,p}[m]} a_h \ind_{[h_{p'}]}^{\Gamma_{p'}} 1$ is a proper character of $\Gamma_{p'}$. So we only have to prove that $\chi(c,g)=a_g\lvert C_G(g)\rvert$ and this follows from
    $$(\ind_{[h]}^{\Gamma} 1)(c,g) = \left\{\matriz{{ll} \lvert C_G(g)\rvert, & \text{if } h\in g^G; \\ 0, & \text{otherwise}.}\right.$$
\end{proof}

\section{Torsion units with $D$-augmentation 1}

In this section $G$ is a finite group. The title of this section refers to $D=Z(C_G(A))$ for a cyclic subgroup $A$ of $G$ containing $G'$ (see the introduction). So the aim of this section is to prove (ZC1) for torsion units $u$ with $\omega_D(u)=1$.

We start with a lemma, which seems to be folklore.

\begin{lemma}\label{DivisoresPrimos}
Let $N$ be a nilpotent normal subgroup of $G$ and $u$ a torsion unit of $\Z G$ such that $\omega_N(u)=1$. Then
\begin{enumerate}
\item every prime divisor of the order of $u$ divides the order of $N$ and
\item if the order of $u$ is a power of a prime $p$, then $\omega_{N_p}(u)=1$.
\end{enumerate}
\end{lemma}

\begin{proof}
(1) We argue by induction on the number of primes dividing $\lvert N\rvert$. If this number is $0$ then $N=1$ and hence $u=\omega_N(u)=1$. Assume that $\lvert N \rvert$ is divisible by $p$.
By hypothesis $\omega_{N/N_p}(\omega_{N_p}(u)) = \omega_N(u)=1$.
Let $n$ be the order of $\omega_{N_p}(u)$.
By the induction hypothesis every prime divisor of $n$ divides $[N:N_p]$.
Moreover the order of $u^n$ is a power of $p$ by \cite[Lemma~7.5]{SehgalBook2}.
Let $q$ be a prime divisor of $\lvert u \lvert$.
Then $q$ is either $p$ or a divisor of $n$. We conclude that $q$ divides $\lvert N\rvert$.

(2) Assume that the order of $u$ is a power of $p$ and set $u_1=\omega_{N_p}(u)$. Then $u_1$ is a $p$-element of $\Z(G/N_p)$ such that $\omega_{N/N_p}(u_1)=\omega_N(u)=1$. Since $p$ is coprime with $[N:N_p]$, the order of $u_1$ is coprime with $p$, by (1). Hence $u_1=1$, as desired.
\end{proof}

The following lemma extends \cite[Claim~5.2]{HertweckEdinb} where it was proved for the case where $G=AX$ with $A$ a cyclic normal subgroup of $G$ and $X$ an abelian subgroup of $G$.

\begin{lemma}\label{pComplemento}
Let $A$ be a cyclic subgroup of $G$ containing $G'$ and let $N$ be a non-trivial $p$-subgroup of $A$ for some prime $p$. Then $\Cen_G(N)$ has a normal $p$-complement (i.e. it has a normal $p'$-Hall subgroup).
\end{lemma}

\begin{proof}
By replacing $G$ by $C_G(N)$ we may assume without loss of
generality that $N$ is central in $G$. If $N_1$ is the unique
minimal non-trivial subgroup of $N$ and $\alpha:\Aut(A_p)\rightarrow
\Aut(N_1)$ is the restriction map, then the kernel of $\alpha$ is a
$p$-group. Therefore the $p'$-elements of $G$ commute with the
$p$-elements of $A$. Let $H$ be a $p'$-Hall subgroup of $G$. We will
prove that $H$ is normal in $G$. Let $h\in H$ and $g\in G$. As $G/A$
is abelian, $h^g=ah$ for some $a\in A$. By Hall Theorem \cite[9.1.7]{RobinsonBook}, it easily
follows, that $A_{p'}\subseteq H$. Therefore
$a_{p'}h\in H$ and in particular the order of  $a_{p'}h$ is coprime
with $p$. Thus $(a_p,a_{p'}h)=1$ and hence the order of $h^g$ is
divisible by the order of $a_p$. Thus $a_p=1$ and we conclude that
$h^g = a_{p'}h\in H$.
\end{proof}

The argument of the following lemma was already used in \cite{delRioSehgal} and \cite{HertweckEdinb} to prove that partial augmentations of elements in $G\setminus C_G(A)$ are non-negative in the minimal counterexamples. We need this also for elements in $G\setminus Z(C_G(A))$.

\begin{lemma}\label{ParcialFueraCentroCentralizador}
Let $A$ be a cyclic subgroup of $G$ containing $G'$ and assume that $G/N$ satisfies (ZC1) for every non-trivial subgroup $N$ of $A$. Let $u$ be a normalized torsion unit in $\Z G$ and let $x\in G\setminus(Z(C_G(A)))$. Then $\varepsilon_x(u)\ge 0$.
\end{lemma}

\begin{proof}
Let $C=C_G(A)$ and $D=Z(C)$. If $x\not\in C$, then $N=\GEN{(A,x)}$ is a non-trivial subgroup of $G$ contained in $A$. By hypothesis, (ZC1) holds for $G/N$. Hence $\varepsilon^{G/N}_{xN}(\omega_N(u))\ge 0$, by Proposition~\ref{ZC1PAugmentation}. By \cite[Lemma~2]{delRioSehgal}, $Nx^G=x^G$ and this implies that $\varepsilon^G_x(u)=\varepsilon^{G/N}_{xN}(\omega_N(u))\ge 0$, as desired.

Assume that $x\in C\setminus D$. Then
$(x,c)\ne 1$ for some $c\in C$. Let $N=\langle(x,c)\rangle$, a
non-trivial normal subgroup of $G$. We claim that $x^G = Nx^G$.
Indeed, as $G'\subseteq A$ and $(A,C)=1$, if $w,v\in C$ and $g\in G$,
then $(wv,g)=(w,g)^v (v,g)=(w,g)(v,g)$. Thus, if $n\in N$, then
$n=(x,c)^m=(x,c^m)$ for some integer $m$. If $g\in G$, then $nx^g = (x,c^m)
x^g = (x,c^m) x (x,g) = x (x,gc^m) =x^{gc^m}$. This proves the
claim.

We set $\bar{\alpha}=\omega_N(\alpha)$ for every $\alpha\in \Z G$.
By hypothesis (ZC1) holds for $G/N$ and hence $\varepsilon_{\bar{x}}^{G/N}(\bar{u})\ge 0$. If $u=\sum_{g\in G} u_g g$ with $u_g\in \Z$ for each $g$ then $\varepsilon_x^G(u)=\sum_{g\in x^G} u_g = \sum_{g\in Nx^G} u_g = \varepsilon_{\bar{x}}^{G/N}(\bar{u})\ge 0$, as desired.
\end{proof}

The following lemma comes from a closer investigation of the proof of Theorem \ref{TCW} as given \cite{CliffWeiss}.

\begin{lemma}\label{Leo}
Let $N$ be an abelian normal subgroup of $G$ and $u$ a torsion unit in $\Z G$ with $\omega_N(u)=1.$ Let $\eta$ be an irreducible character of $N$ and $n\in N$. Then $$\sum_{h\in \ker \eta} [C_G(hn):N]\varepsilon^G_{hn}(u)\ge 0.$$
\end{lemma}

\begin{proof}
%
Let $m$ be the order of $u$ and let $v$ be the image of $u$ under the natural permutation homomorphism $\mathbb{Z}G \rightarrow GL_k(\mathbb{Z}N)$.
Then $v \in SGL_k(\mathbb{Z}N)$, where $k = [G:N]$. Let $\Gamma=C_m\times N$, with $C_m=\GEN{c}$, a cyclic group of order $m$, let $\chi$ be the character of $M^{\alpha_v}$ and let $N[m]=\{n\in N : n^m=1\}$.
%
%
%
By \cite[Lemma 1]{WeissCrelle} we know $\lvert C_G(n) \rvert \varepsilon_n^G(u)= \chi(c,n)$ for every $n\in N$. Moreover, $\varepsilon_h^G(u)=0$ if $h\in N\setminus N[m]$, by statement~(\ref{AugmentationOrder}) of Proposition~\ref{Hertweck}.
Combining this with Lemma~\ref{character} we have
    \begin{eqnarray*}
    \chi &=& \sum\limits_{h \in N[m]} \frac{\chi(c,h)}{\lvert C_N(h) \rvert} {\rm ind}_{[h]}^\Gamma 1 =
    \sum\limits_{h \in N[m]} [C_G(h):N] \varepsilon_h^G(u) {\rm ind}_{[h]}^\Gamma 1 \\
    &=& \sum\limits_{h \in N} [C_G(h):N] \varepsilon_h^G(u) {\rm ind}_{[h]}^\Gamma 1
    \end{eqnarray*}
and for every prime integer $p$ and every character $\psi$ of $\Gamma_{p'}$ we have
    \begin{eqnarray*}
    0 &\leq &
    \left\langle \sum\limits_{h \in N_{n,p}[m]} [C_G(h):N] \varepsilon_h^G(u) {\rm ind}_{[h_{p'}]}^{\Gamma_{p'}} 1, \psi \right\rangle _{\Gamma_{p'}} \\
    &=&
    \sum\limits_{h \in N_{n,p}[m]} [C_G(h):N] \varepsilon_h^G(u) \langle {\rm ind}_{[h_{p'}]}^{\Gamma_{p'}} 1, \psi \rangle _{\Gamma_{p'}} \\
    &=& \sum\limits_{h \in N,h_p=n_p,\psi(c_{p'},h_{p'})=1} [C_G(h):N] \varepsilon_h^G(u)
    = \sum\limits_{h \in N_{p'}, \psi(1,h)=1} [C_G(hh'):N] \varepsilon_{hh'}^G(u),
    \end{eqnarray*}
where $h'$ is a fixed element of $N$ with $h'_p=n_p$ and $\psi(c,h')=1$.

Let $\psi$ be the character of $\Gamma$ given by $\psi|_N=\eta$ and $\psi(u)=\eta(n)\inv$. Then $\psi(c,n)=1$ and therefore applying the previous inequality for $h'=h_pn$ with $h_p\in N_p$ we deduce that
    $$\sum\limits_{h \in N_{p'}\cap \ker \eta} [C_G(hh_pn):N] \varepsilon_{hh_pn}^G(u)\ge 0.$$
We conclude that
    $$\sum\limits_{h\in \ker\eta} [C_G(hn):N] \varepsilon_{hn}^G(u) = \sum\limits_{h_p\in N_p\cap \ker\eta} \sum_{h_{p'}\in N_{p'}\cap \ker\eta} [C_G(h_{p'}h_pn):N] \varepsilon_{h_{p'}h_pn}^G(u) \ge 0,$$
as desired.
\end{proof}

The following theorem is an adjustment of \cite[Theorem~5.1]{HertweckEdinb} to our situation.\\
\begin{theorem}\label{D_pConjugation}
Let $G$ be a finite group and $A$ a cyclic normal subgroup of $G$ containing $G'$. Set $D = Z(C_G(A))$ and let $u$ be a torsion unit of $\Z G$ with $\omega_D(u)=1$. If the order of $u$ is a power of a prime $p$, then $u$ is conjugate in $\mathbb{Z}_pG$ to an element of $D_p$.
\end{theorem}

\begin{proof}
By statement~(\ref{pAugmentation}) of Proposition~\ref{Hertweck}, $u$ is conjugate  in $\mathbb{Q}G$ to an $x \in D_p$.
Let $R$ be a $p$-adic ring with quotient field $K$ containing a root of unity of order the exponent of $G$.
We will proof that $u$ is conjugate to $x$ in $RG$. Then by \cite[30.25]{CurtisReiner} the conjugation already takes place in $\mathbb{Z}_pG.$

Set $L = C_G(D_p), \ E=C_G(x)$ and let $Q$ be the normal $p$-complement of $L$, that exists by Lemma~\ref{pComplemento}. Note that $L$ and $E$ are normal in $G$, for they contain $A$, and $Q$ is also normal in $G$ for it is a characteristic subgroup of $L$.

The primitive central idempotents of $KQ$ belong to $RQ$, because the order of $Q$ is invertible in $R$. Moreover $G$ acts on these primitive central idempotents by conjugation. Let $\epsilon_1,...,\epsilon_{\beta}$ be the sums of the $G$-orbits of this action. Then $RG = \prod\limits_{i=1}^\beta \epsilon_iRG$ and therefore it is enough to show that $\epsilon_iu$ is conjugate to $\epsilon_ix$ in $\epsilon_iRG$ for every $i$. Note that $\epsilon_iu$ is conjugate to $\epsilon_ix$ in $\epsilon_iKG$ and a primitive idempotent of $KQ$ stays primitive in $KL$ by Greens Indecomposability Theorem, since $L/Q$ is a $p$-group \cite[19.23]{CurtisReiner}.

So fix one primitive central idempotent $f$ of $KQ$ and let $\epsilon$ be the sum of the $G$-conjugates of $f$. We have to prove that $\epsilon u$ and $\epsilon x$ are conjugate in $\epsilon RG$.
Let $e$ be the sum of different $L$-conjugates of $f$ and write $e = e_1 + ... + e_m$ with orthogonal primitive idempotents of $RQ$.
Let $T=C_G(e)$ and let $\{1=s_1,...,s_n\}$ be a transversal of $G/T$.
Then
$$\epsilon = \sum\limits_{i=1}^m \sum\limits_{j=1}^n e_i^{s_j^{-1}}$$
is a decomposition into primitive orthogonal idempotents of $\epsilon RL$.
So by \cite[Lemma 4.6]{HertweckEdinb} there exist $g_{ij} \in G$ such that
$$v = \sum\limits_{i=1}^m \sum\limits_{j=1}^n e_i^{s_j^{-1}} x^{g_{ij}}$$
is conjugate to $\epsilon u$ in $RG$. In particular $|v|=|\epsilon u|=|\epsilon x|$. Let $C=\GEN{c}$ a cyclic group with the same order as $v$ and consider the $R(C\times G)$-modules $M=M^{\alpha}$ and $N=M^{\beta}$, with $\alpha=\alpha_v$ and $\beta=\alpha_{\epsilon x}$. (We are abusing the notation by writing $\alpha_v$ and $\alpha_{\epsilon x}$ instead of $\alpha_{v+(1-\epsilon)}$ and $\alpha_{\epsilon x + (1-\epsilon)}$, respectively.) We have to show that $v$ is also conjugate to $\epsilon x$ in $\epsilon RG$ or equivalently that $M$ and $N$ are isomorphic as $R(C\times G)$-modules.

%
%
As $x\in D_p$ and $e_i\in RL$ and both $D_p$ and $L$ are normal in $G$, every $G$-conjugate of $x$ belongs to $D_p$ and every $G$-conjugate of $e_i$ belong to $RL$. On the other hand $(D_p,L)=1$ and therefore every $G$-conjugate of $x$ commutes with every $G$-conjugate of $e_i$.
Using this it easily follows that each $RGe_{i}^{s_j\inv}$ is a submodule of both $M$ and $N$. We set $M_{ij} = Me_i^{s_j^{-1}}$ and $N_{ij}=Ne_i^{s_j^{-1}}$, i.e. both $M_{ij}$ and $N_{ij}$ are $RGe_{i}^{s_j\inv}$, considered as submodules of $M$ and $N$ respectively. The strategy of the proof consists in pairing isomorphic $M_{ij}$'s and $N_{ij}$'s.

Firstly observe that if $g\in G$, then every two primitive idempotents $\varepsilon_1$ and $\varepsilon_2$ of $RQe^g$ are conjugate in $RLe^g$, since $KLe^{g}$ is simple and hence $RL\varepsilon_1\cong RL \varepsilon_2$ (see e.g. Theorem 6.7, Proposition 16.16 and Problem~6.14 in \cite{CurtisReiner}). If $e_i^{s_j\inv}=e_1^w$ with $w$ a unit in $RQe$ then $a\mapsto aw$ is an isomorphism $N_{1j}\rightarrow N_{ij}$. Secondly, if $q\in E$, then the map $a\mapsto aq$ is an isomorphism $Ne_i\rightarrow Ne_i^q$. Therefore, if we choose the transversal $\{ s_j \ : \ 1 \leq j \leq n\} = \{ h_{j_2}q_{j_1} \ | \ 1 \leq j_1 \leq n_1, \ 1 \leq j_2 \leq n_2\}$, with $\{h_1,...,h_{n_2}\}$ a transversal of $G/TE$ and $\{q_1,...,q_{n_1}\}$ a transversal of $E/E\cap T$ (which is also a transversal of $TE/T$), and denote $N_{ij}=N_{i(j_1,j_2)}$ if $s_j=h_{j_2}q_{j_1}$. Then we have
\begin{equation}\label{N}
N = \bigoplus\limits_{i,j}N_{ij} \cong \bigoplus\limits_{j_2=1}^{n_2} \left(N_{1(1,j_2)}\right)^{mn_1}
\end{equation}
If we have $g_{ij}s_j \equiv s_{j_0} \mod T$ then $N_{ij_0}\cong M_{ij}$ via $a \rightarrow ag_{ij}$. So if $g_{ij}s_j \equiv h_{j_2} \mod TE$ we can pick a suitable $q_{j_1}$ such that by setting $s_{j_0} = h_{j_2}q_{j_1}$ we have $g_{ij}s_j \equiv s_{j_0} \mod T$ and this gives $M_{ij} \cong N_{1,(1,j_2)}$.

Set $X_{j_2} = \{(i,j) \ | \ g_{ij}s_j \equiv h_{j_2} \mod TE\}$.
By the previous paragraph, the isomorphism $M\cong N$ will follow from (\ref{N}) provided $\lvert X_{j_2} \rvert = mn_1$. The remainder of the proof is dedicated to prove this equality. For this we will use a representation of $\epsilon KG$ and investigate the multiplicities of eigenvalues of $\epsilon x$ and $v$ under this representation. They are the same for $\epsilon x$ and $v$ are conjugate in $\epsilon KG$. This is also the strategy in the proof of \cite[Theorem~5.1]{HertweckEdinb}. However the representation used by Hertweck was constructed using a faithful linear representation of $A_p$ and here we need to use a linear representation of $D_p$. This representation cannot be faithful if $D_p$ is not cyclic. Instead we use a subgroup $H$ of $G$ such that $G/H$ is cyclic and $H\cap A=1$ and then consider a linear representation of $D_p$ with kernel $H$. The existence of such $H$ follows easily by observing that if $H$ is a maximal subgroup of $D_p$ not intersecting $A$, then $D_p/H$ is cyclic because otherwise $D_p/H$ contains a direct product $\GEN{c}\times \GEN{d}$ of cyclic groups of order $p$. This would yield to a contradiction using the maximality of $H$ and the fact that $A$ is cyclic.

Let $\pi$ be the linear character of a representation of $D_p$ with kernel $H$ and let $\psi$ be the sum of all irreducible characters of $Q$, that do not vanish on $e$. So $\psi(f) = 1$ for every primitive idempotent $f$ of $RQ$ satisfying $ef \neq 0$ and hence $\psi(1)=m$. Let $\rho$ be a representation of $Q \times D_p$ affording $\psi \otimes \pi$. This can be chosen satisfying $\rho(e_i)=E_i$, where $E_i$ denotes the elementary matrix with 1 in the $i$-th diagonal entry and 0 anywhere else. Then $\rho(e_iy)=\pi(y)E_{i}$ for any $i$ and $y \in D_p$. Let $\chi = { \rm ind}_{Q\times D_p}^T (\psi \otimes \pi)$, let $\Delta$ be a representation of $T$ affording $\chi$ and let $\{t_
1,...,t_k\}$ be a transversal of $T/ {Q \times D_p}$. Then, after a suitable conjugation one may assume that
    $$\Delta\left(\sum\limits_{i=1}^m e_i y_i\right) = \diag\left(\pi\left(y_i^{t_j}\right) \ | \ 1 \leq i \leq m, \ 1 \leq j \leq k\right)\in M_{mk}(K)$$
for every $y_1,\dots,y_m \in D_p$
Denote by $\bar{\Delta}: {\rm M}_n(eKT) \rightarrow {\rm M}_{nmk}(K)$ the map which acts like $\Delta$ componentwise.

As right $KG$-modules we have $\epsilon KG = \sum\limits_{j=1}^n e^{s_j^{-1}} KG \cong (eKG)^n$ and so
\begin{equation}\label{delta}
\epsilon KG \cong {\rm End}_{KG}(\epsilon KG) \cong {\rm M}_n({\rm End}_{KG}(eKG)) \cong {\rm M}_n(eKGe).
\end{equation}
Moreover, $eKGe = e \sum\limits_{j=1}^n s_jKTe = e \sum\limits_{j=1}^n e^{s_j^{-1}}s_j^{-1}KT = eKT$, since $e$ is orthogonal to any different conjugate of itself.
So (\ref{delta}) gives an isomorphism $\delta:\epsilon KG \cong {\rm M}_n(eKT)$, which satisfies $\delta(\sum\limits_{j=1}^n e^{s_j^{-1}}y_j) = \diag(ey_1^{s_1},...,ey_n^{s_n})$ for $y_1,\dots,y_k\in KT$.  Set $\hat{\Delta} = \bar{\Delta} \circ \delta$. Then we have
\begin{equation}\label{D(x)}
\hat{\Delta}(\epsilon x) = \bar{\Delta}\left(\diag\left(ex^{s_1},...,ex^{s_n}\right)\right) = \diag\left(\pi\left(x^{s_jt_l}\right) : 1 \leq i \leq m, \ 1 \leq j \leq n, \ 1 \leq l \le k\right). \nonumber
\end{equation}
Observe that the index $i$ only affects the diagonal entries of $\hat{\Delta}(\epsilon x)$ by repeating each entry $m$ times. Furthermore $\{s_jt_l : 1\le j \le n, 1 \leq l \le k\}$ is a transversal of $G/(Q\times D_p)$.
Note that if $x^g \neq x$, then $x^g = ax$ with $1\ne a \in A$ and so $\pi(x^g) = \pi(a)\pi(x) \neq \pi(x)$ by the choice of $\pi$. So, the diagonal of $\hat{\Delta}(\epsilon x)$ is formed by $[G:E]$ different entries each with multiplicity $m[E:Q\times D_p]$.

On the other hand we have
\begin{eqnarray}
\hat{\Delta}(v) &=& \hat{\Delta}\left(\sum\limits_{i=1}^m \sum\limits_{j=1}^n e_i^{s_j^{-1}}x^{g_{ij}}\right) = \bar{\Delta}\left(\diag\left(\sum\limits_{i=1}^m e_ix^{g_{i1}s_1},...,\sum\limits_{i=1}^m e_i x^{g_{in}s_n}\right)\right) \nonumber \\ \nonumber
&=& \diag\left(\pi\left(x^{g_{ij}s_jt_l}\right) \ | \ 1\leq i \leq m, \ 1 \leq j \leq n, \ 1 \leq l \leq k\right)
\end{eqnarray}
Let $\{v_1,\dots,v_r\}$ be a transversal of $T/T\cap E$ and set $Y_{j_2}=\{(i,j,l) : g_{ij}s_jt_l \equiv h_{j_2} \mod E\}$ and $\bar{Y}_{j_2}=\{(i,j,l) : g_{ij}s_jv_l \equiv h_{j_2} \mod E\}$. Arguing as in the previous paragraph we deduce that the multiplicity of $\pi(x^{h_{j_2}})$ in $\hat{\Delta}(v)$ is $\lvert Y_{j_2} \rvert = [T\cap E:Q\times D_p] \lvert \bar{Y}_{j_2} \rvert$.
As $v$ and $\epsilon x$ are conjugate in $\epsilon KG$ we deduce that $m[E:Q\times D_p]=[T\cap E:Q\times D_p] \lvert \bar{Y}_{j_2} \rvert$ and therefore $|\bar{Y}_{j_2}|=m[E:Q\times D_p]/[T\cap E:Q\times D_p] = m [E:T\cap E] = mn_1$. Then $(i,j,l)\mapsto (i,j)$ defines a bijective map $\bar{Y}_{j_2} \rightarrow X_{j_2}$. Therefore $\lvert X_{j_2} \rvert = m n_1$ as desired.
\end{proof}

\begin{corollary}\label{OmegaD}
Let $G$ be a finite group and $A$ a normal cyclic subgroup of $G$ containing $G'$. Let $D=Z(C_G(A))$ and $u$ a torsion unit in $\Z G$ satisfying $\omega_D(u)=1.$ Then $u$ is conjugate in $\Q G$ to an element of $D$.
\end{corollary}

\begin{proof}
Since $G/A$ is abelian there exists some $b\in G$ such that $\omega_A(u)=bA$. We claim that $\{n\in D:\varepsilon_n(u)\neq 0\}$ is contained in $bA$.
Indeed, let $p$ be some prime dividing $|u|$. By Theorem~\ref{D_pConjugation} $u_p$ is conjugate in $\Z_pD$ to some $n_{p,0}\in D$. By statement~(\ref{pPartesNoConjugadas}) of Proposition \ref{Hertweck}, if $\varepsilon_n(u)\ne 0$ then $n_p$ is conjugate to $n_{p,0}$, so $n_pA=n_{p,0}A$. On the other hand we have that $b_pA=\omega_A(u_p)$ is conjugate to $n_{p,0}A=\omega_A(n_{p,0})$,  so $b_pA=n_{p,0}A$ and hence $b_pA=n_pA$ for all primes $p$ dividing $|u|$ and all $n\in N$ with $\varepsilon_n(u)\ne 0$. As $|n|$ divides $|u|$, by statement~(\ref{AugmentationOrder}), we deduce that $nA=bA$ and the claim is proved.

Assume that the statement of the corollary is false. Then, by Proposition~\ref{ConjugatePartialAugmentation} and  Lemma~\ref{ParcialFueraCentroCentralizador}, there exists an $n \in G$ such that $\varepsilon_{n}(u) < 0$. So $n\in bA$, by the previous paragraph. As in the proof of Theorem \ref{D_pConjugation} there exists a subgroup $H$ in $D$ such that $H\cap A=1$ and $D/H$ is cyclic. So $D$ has a linear character $\eta$ with kernel $H$. Then, by Lemma~\ref{Leo}, we have $0 \le \sum_{h\in H} [C_G(hn):D]\varepsilon_{hn}(u)$. But if $\varepsilon_{hn}(u)\neq 0$ with $h\in H_{p'}$ then $hn\in bA\cap nH=nA\cap nH=n(A\cap H)=\{n\}$.
Hence $$0\le \sum_{h\in H} [C_G(hn):D]\varepsilon_{hn}(u)=[C_G(n):D]\varepsilon_{n}(u)<0,$$ a contradiction.
%
\end{proof}

\section{Reduction to torsion units of $D$-augmentation 1}

In this section we pursue another idea from \cite{HertweckEdinb}. In order to reduce the proof of (ZC1) to the case of units of $A$-augmentation 1, with $A$ a cyclic normal subgroup of $G$, Hertweck studied the multiplicities of the image of a representation of $G$ induced from a faithful linear character of $A$. In our study we need to replace $A$ by $D=Z(C_G(A))$. Now $D$ may not be cyclic and in that case it does not have any faithful linear character. Alternatively we consider linear characters of $D$, whose kernel does not contain any non-trivial normal subgroup of $G$. Observe that if $D$ is cyclic then the characters satisfying this condition are precisely the faithful linear characters of $D$.

Let $N$ be an abelian normal subgroup of $G$. Then for every linear character $\psi$ of $N$ and every $u\in \C G$ one has
    \begin{equation}\label{FormulaCaracteres}
    \psi^G(u) = \sum_{n\in N} \psi(n)\; [C_G(n):N]\; \varepsilon^G_n(u),
    \end{equation}
where $\psi^G$ represents the character  induced from $\psi$.
This can be checked directly by observing that both sides of the equality define linear maps on $\C G$ and checking the formula for the elements of $G$. It can be also proved using \cite[Lemma~41.10]{SehgalBook2}. If $N=\{n_1,\dots,n_m\}$ and $\psi_1,\dots,\psi_m$ are the linear characters of $N$ then (\ref{FormulaCaracteres}) yields the following
    $$\pmatriz{{c} \psi_1^G(u) \\ \vdots \\ \psi_m^G(u)} =
    T \pmatriz{{c} [C_G(n_1):N] \; \varepsilon^G_{n_1}(u) \\ \vdots \\ { } [C_{G}(n_m):N] \; \varepsilon^G_{n_m}(u)},$$
where $T=(\psi_i(n_j))$, the character table of $N$. By the Orthogonality Relations the transpose conjugate of $T$ is $mT\inv$. Multiplying by this matrix we obtain
    \begin{equation}\label{FormulaCaracteres2}
    |C_G(x)|\varepsilon_x^G(u) = \sum_{i=1}^m \overline{\psi_i(x)} \psi_i^G(u) \quad  \text{ for every } x\in N \text{ and } u\in \C G.
    \end{equation}

Let
    \begin{eqnarray*}
    \K=\K_N&=&\{K\le N : N/K \text{ is cyclic and } \\ && K \text{ does not contain any non-trivial normal subgroup of } G\}. \end{eqnarray*}
For every $K\in \K$ we select a linear character $\psi_K$ of $K$ with kernel $K$ and let $\rho_K$ be a representation of $G$ affording the induced character $\psi_K^G$. Observe that if $K_1$ and $K_2$ are conjugate in $G$ then $\psi_{K_1}^G=\psi_{K_2}^G$ and therefore we may assume that $\rho_{K_1}=\rho_{K_2}$. Let $\mathcal{C}_{\K}$ be the set of conjugacy classes in $G$ of elements of $\K$. For every $C\in \mathcal{C}_{\K}$ select a representative $K_C$ of $C$.
Let $\Q_K=\Q(\zeta_{[N:K]})$.

For a square matrix $U$ with entries in $\C$ and $\alpha\in \C$ let $\mu_U(\alpha)$ denote the multiplicity of $\alpha$ as eigenvalue of $U$. If $U^m=I$ and $\alpha$ is a root of unity then we have the following formula (see \cite{LutharPassiA5})
    \begin{equation}\label{Multiplicidad}
    \mu_U(\alpha)=\frac{1}{m} \sum_{d\mid m} \tr_{\Q(\zeta_m^d)/\Q}(\tr(U^d)\alpha^{-d}).
    \end{equation}
This formula is the bulk of the Luthar-Passi Method.

\begin{lemma}\label{PartialAugmentationTraces}
Let $G$ be a finite group such that (ZC1) holds for every proper quotient of $G$. Let $N$ be an abelian normal subgroup of $G$. Let $u$ be a unit of $\Z G$ with $\omega_N(u)\ne 1$ and let $x\in N$.

\begin{enumerate}
\item
Then
\begin{equation}\label{FormulaCaracteresK}
    |C_G(x)|\varepsilon_x(u) = \sum_{K\in \K} \tr_{\Q_K/\Q}(\overline{\psi_K(x)} \psi_K^G(u)).
\end{equation}

\item Assume moreover that $m=|u|$, $f=|\omega_N(u)|$, $x^m=1$ and $u^d$ is conjugate in $\Q G$ to an element of $G$ for every $1\ne d| m$. Then for every $h\mid f$ with $h\ne 1$ we have
    \begin{eqnarray}\label{KsPartialAugmentation}
    \sum_{K\in \K} [\Q_K:\Q] \; \mu_{\rho_K(u)}(\psi_K(x)) &=& \frac{\varphi(m)}{m} |C_G(x)| \varepsilon_{x}(u)+\\ &&\frac{1}{h}\sum_{K\in \K} [\Q_K:\Q] \; \mu_{\rho_K(u^h)}(\psi_K(x^h)) \nonumber
    \end{eqnarray}

\item Assume moreover that $G'$ is cyclic and $u^f$ is conjugate in $\Q G$ to an element $y$ of $N$. Let $u_C=|\{g\in G : x^f\in y^GK^g_C\}|$ for $C\in \mathcal{C}_{\K}$. Then we have
\begin{equation}\label{Dmultiplicidades}
\sum_{K\in C} \mu_{\rho_K(u^f)}(\psi_K(x^f))= \frac{[C_G(y):N]}{| N_G(K_C)|} u_C.
\end{equation}
\end{enumerate}
\end{lemma}

\begin{proof}
(1) Let $\psi$ be a linear character of $N$ such that the kernel of $\psi$ contains a non-trivial normal subgroup $U$ of $G$. Then $\psi=\phi\circ \omega_U$ for a linear character $\phi$ of $G/U$. By the induction hypothesis $\omega_U(u)$ is conjugate in $\Q(G/U)$ to an element of $G/U$. Moreover $\omega_{N/U}(\omega_U(u))=\omega_N(u)\ne 1$  and therefore $\varepsilon_{nU}(\omega_U(u))=0$ for every $n\in N$. Then (\ref{FormulaCaracteres})  yields $\psi^G(u)=\phi^G(\omega_U(u))=0$. Hence we can drop in (\ref{FormulaCaracteres2}) all the summands labeled by linear characters of $N$ whose kernel is not in $\K$. The remaining characters are those of the form $\sigma\circ \psi_K$ for a $K\in \K$ and $\sigma\in \Gal(\Q_K/\Q)$. Hence
    $$|C_G(x)|\varepsilon_x(u) = \sum_{K\in \K} \sum_{\sigma\in \Gal(\Q_K:\Q)} \overline{\sigma \circ \psi_K(x)} (\sigma \circ \psi_K^G)(u) = \sum_{K\in \K} \tr_{\Q_K/\Q}(\overline{\psi_K(x)} \psi_K^G(u)),$$
as desired.

(2) Let $d$ be a divisor of $m$ such that $d\ne 1$ and $\omega_N(u^d)\ne 1$. By hypothesis $u^d$ is conjugate to an element of $G$ which does not belong to $N$, by Remark~\ref{ZC1PAugmentation}. Therefore $\varepsilon_n(u^d)=0$ for every $n\in N$. Hence, by (\ref{FormulaCaracteresK}) we have
$\sum_{K\in \K} \tr_{\Q_K/\Q}((\psi_K^G)(u^d)\; \psi_K(n)\inv)=0$.
Thus,
    \begin{equation}\label{SumaCero}
    \text{if } f\nmid d| m,  d\ne 1 \text{ and } n\in N \text{ then } \sum_{K\in \K} \tr_{\Q_K/\Q}((\psi_K^G)(u^d)\; \psi_K(n)\inv)=0.
    \end{equation}

For every $K\in \K$ and every integer $d$ we use the notation $\mu(K,d)=\mu_{\rho_K(u^d)}(\psi_K(x^d))$, the multiplicity of $\psi_K(x^d)$ as an eigenvalue of $\rho_K(u^d)$.
By (\ref{Multiplicidad}), for every $e|m$ we have
    \begin{eqnarray}\label{MultiplicidadesTrazas}
    &&\sum_{K\in \K} [\Q_K:\Q] \; \mu(K,e) = \\&& \frac{e}{m} \sum_{K\in \K} \sum_{d| (m/e)} [\Q_K:\Q] \; \tr_{\Q(\zeta_m^{ed})/\Q}(\psi_K^G(u^{ed})\psi_K(x)^{-ed}). \nonumber
    \end{eqnarray}
Let $d | m$ and $\alpha=\psi_K^G(u^d)\psi_K(x)^{-d}$. Clearly the image of $\psi_K$ is $\Q_K$ and $\psi^G_K(u^d)\in \Q(\zeta_m^d)$. Moreover, $x^m=1$, by hypothesis. Thus $\alpha\in \Q_K\cap \Q(\zeta_m^d)$. Let $L=\Q_K(\zeta_m^d)$. Then $[L:\Q_K]\tr_{\Q_K/\Q}(\alpha) = (\tr_{\Q_K/\Q}\circ \tr_{L/\Q_K})(\alpha) = \tr_{L/\Q}(\alpha) = (\tr_{\Q(\zeta_m^d)/\Q}\circ \tr_{L/\Q(\zeta_m^d)})(\alpha) = [L:\Q(\zeta_m^d)] \tr_{\Q(\zeta_m^d)/\Q}(\alpha)$. Therefore $[\Q(\zeta_m^d):\Q]\tr_{\Q_K/\Q}(\alpha) = [\Q_K:\Q] \tr_{\Q(\zeta_m^d)/\Q}(\alpha)$. This equality together with (\ref{FormulaCaracteresK}) yields
    \begin{eqnarray*}
    &&\sum_{K\in \K} [\Q_K:\Q] \tr_{\Q(\zeta_m^d)/\Q}(\psi_K^G(u^d)\;\psi_K(x)^{-d}) = \\
    &&[\Q(\zeta_m^d):\Q] \sum_{K\in \K} \tr_{\Q_K/\Q}(\psi_K^G(u^d)\;\psi_K(x)^{-d}) = \\
    && [\Q(\zeta_m^d):\Q] |C_G(x^d)|\varepsilon_{x^d}(u^d).
    \end{eqnarray*}
Moreover, by (\ref{SumaCero}), this is $0$ provided $f\nmid d| m$ and $d\ne 1$. Thus, for $e=1$ and $1\ne h| f$, (\ref{MultiplicidadesTrazas}) can be reduced to the following
\begin{eqnarray*}
&&\sum_{K\in \K} [\Q_K:\Q] \; \mu(K,1) =\\
&& \frac{[\Q(\zeta_m):\Q] |C_G(x)|}{m} \varepsilon_{x}(u)+\frac{1}{m}\sum_{K\in \K} \sum_{h| d | m} [\Q_K:\Q] \; \tr_{\Q(\zeta_m^d)/\Q}(\psi_K^G(u^d)\psi_K(x)^{-d}) = \\
&& \frac{\varphi(m) |C_G(x)|}{m} \varepsilon_{x}(u)+\frac{1}{h}\frac{h}{m}\sum_{K\in \K} \sum_{d | (m/h)} [\Q_K:\Q] \; \tr_{\Q(\zeta_m^{hd})/\Q}(\psi_K^G(u^{hd})\psi_K(x)^{-hd}) =\\
&& \frac{\varphi(m) |C_G(x)|}{m} \varepsilon_{x}(u)+\frac{1}{h}\sum_{K\in \K} [\Q_K:\Q] \; \mu(K,h),
\end{eqnarray*}
where in the last equality we have used (\ref{MultiplicidadesTrazas}) for $e=h$. This proves (\ref{KsPartialAugmentation}).

(3) Finally assume that $G'$ is cyclic and $u^f$ is conjugate in $\Q G$ to $y\in N$. Then $\rho_K(u^f)$, $\rho_K(y)$ and $\diag(\psi_K(y^g):g\in T)$ are conjugate in the matrices over $\C$, where $T$ is a transversal of $G/N$. Observe that $\psi_K(y^g)=\psi_K(y^h)$ if and only if $\psi_K((y,g))=\psi_K((y,h))$ if and only if $(y,g)(y,h)\inv \in K$ if and only if $(y,g)=(y,h)$ (because $K\cap G'=1$), if and only if $gh\inv \in C_G(y)$. Therefore each eigenvalue of $\rho_K(u^f)$ has multiplicity $[C_G(y):N]$. On the other hand $\psi_K(x^f)$ is an eigenvalue of $\rho_K(u^f)$ if and only if $\psi_K(x^f)=\psi_K(y^g)$ for some $g\in G$, if and only if $x^f \in y^GK$. Therefore if $C\in \mathcal{C}_{\K}$ and $K_C$ is a representative of $C$ then
    $$\sum_{K\in C} \mu_{\rho_K(u^f)}(\psi_K(x^f)) = \frac{1}{|N_G(K_C)|} \sum\limits_{g\in G, x^f \in y^G K_C^g} [C_G(y):N] = \frac{[C_G(y):N]}{| N_G(K_C)|} u_C,$$
as desired.
%
%
\end{proof}

\begin{remark}\label{FormaK}
{\rm Let $A$ be a cyclic normal subgroup of $G$ containing $G'$. Clearly every element of $\K$ does not intersect $A$ and $Z(G)$. Conversely, let $H$ be a subgroup of $G$ containing a non-trivial normal subgroup $U$ of $G$ and such that $H\cap Z(G)=1$. If $1\neq n \in U$ then $1\neq (n,g)\in A\cap U$ for some $g\in G$ and therefore $A\cap H\neq 1$. Thus, for every abelian subgroup $N$ of $G$ we have $\K_N=\{K\le N : A\cap K = Z(G)\cap K = 1 \text{ and } N/K \text{ is cyclic}\}$.

Observe that $\K_N$ can be empty. For example, this is the case if $N\cap Z(G)$ is not cyclic.}
\end{remark}

\begin{lemma}\label{Nucleos}
Assume that $A$ is a cyclic subgroup of $G$ containing $G'$. Let $N$ be a an abelian subgroup of $G$ containing $A$ and $\K=\K_N$. Then for every $K\in \K$ we have $|\K|\le |K|=\frac{|N|}{exp(N)}$.
\end{lemma}

\begin{proof}
Write $N=C\times H$ with $C$ cyclic of maximal order in $N$ and selected in such a way that if $p$ is prime and $\exp(N_p)=\exp(A_p)$, then $C_p=A_p$. We claim that if $K\in \K$, then $C\cap K = 1$. Otherwise $C_p\cap K\ne 1$ for some prime $p$ and therefore $\exp(C_p)=\exp(N_p)>\exp(A_p)$. Let $x$ be a generator of $C_p$, $q=|A_p|$ and $a=(x,g)$ with $g\in G$. Then $a\in A_p$ and therefore $a^q=1$. Thus $(x^q)^g = x^q$. This proves that $x^q$ is a non-trivial central element of $G$. Then $Z(G)\cap K\ne 1$, contradicting the fact that $K$ does not contain any normal subgroup of $G$. This proves the claim.

Let $\pi_1$ and $\pi_2$ be the projections $N\rightarrow C$ and $N\rightarrow H$ along the decomposition $N=C\times H$. By the previous paragraph $K\cap \ker \pi_2 = 1$ and therefore $|K|\ge |H|=\frac{|N|}{\exp(N)}$. As $N/K$ is cyclic we have $\exp(N)\le [N:K]=\exp(N/K)\le \exp(N)$. Hence $|K|=\frac{|N|}{\exp(N)}=|H|$ and therefore $\pi_2|_K:K\rightarrow H$ is an isomorphism for every $K\in \K$. Therefore $K=\{f(h)h : h\in H\}$ for a homomorphism $f:H\rightarrow C$. (More precisely $f=\pi_1\circ \pi_2|_K\inv$.) Thus $K$ is completely determined by $f$ and hence $|\K|\le |\Hom(H,C)| = |H|$. The last equality follows easily from the fact that $C$ is cyclic and $\exp(H)$ divides the order of $C$.
%
%
%
%
%
%
\end{proof}

We are ready to prove our main result.

\begin{theorem}\label{OmegaNotD}
If $G$ is a cyclic-by-abelian finite group then every normalized torsion unit of $\Z G$ is conjugate in $\Q G$ to an element of $G$.
%
%
\end{theorem}

\begin{proof}
By means of contradiction we assume that $G$ is a counterexample of minimal order of the theorem and $u$ is a normalized torsion unit of minimal order of $\Z G$ which is not conjugate to an element of $G$ in $\Q G$. We select a cyclic subgroup $A$ of $G$ with $G/A$ abelian and take $D=Z(C_G(A))$ and $\K=\K_D$. By Proposition~\ref{ConjugatePartialAugmentation}, we may assume without loss of generality that $\varepsilon_x(u)<0$ for some $x\in G$. This implies that the order of $x$ divides the order of $u$ by statement~(\ref{AugmentationOrder}) of Proposition~\ref{Hertweck}. Set $m=|u|$ and $f=|\omega_D(u)|$. By assumption $u^d$ is conjugate in $\Q G$ to an element of $G$ for every $1\ne d\mid m$.

By Lemma \ref{ParcialFueraCentroCentralizador}, $x\in D$ and by Corollary~\ref{OmegaD}, $\omega_D(u)\ne 1$. Thus $1\ne f|m$ and in particular, $u^f$ is conjugate in $\Q G$ to some $y\in D$. By the first induction hypothesis (ZC1) holds for every proper quotient of $G$ and hence we can use Lemma \ref{PartialAugmentationTraces} for $N=D$, $u$ and $x$. Since $|K|=\frac{|D|}{\exp(D)}$ for every $K\in \K$, by Lemma \ref{Nucleos}, and $[\Q_K:\Q]=\varphi([D:K])$ we can write (\ref{KsPartialAugmentation}) for $h=f$ as
\begin{equation}\label{KsPartialAugmentation2}
    \sum_{K\in \K} \mu_{\rho_K(u)}(\psi_K(x)) = \frac{\varphi(m)}{m} \frac{|C_G(x)|}{\varphi([D:K])} \varepsilon_{x}(u)+\frac{1}{f}\sum_{K\in \K}  \mu_{\rho_K(u^f)}(\psi_K(x^f)).
    \end{equation}

We claim that
\begin{equation}\label{Desigualdad}
\frac{1}{f}\sum_{k\in \K}\mu_{\rho_K(u^f)}(\psi_K(x^f))\le \frac{\varphi(m)}{m}\frac{| C_G(x)|}{\varphi( [D:K])}.
\end{equation}
Write $f=f_1f_2$ with $f_1$ and $f_2$ positive integers such that the prime divisors of $f_1$ divide $|D|$ and $(f_2,|D|)=1$. Then $m=f_2m'$ with all prime divisors of $m'$ dividing $|D|$. Note that $\langle x^f \rangle=\langle x^{f_1}\rangle$ and so $C_G(x^f)=C_G(x^{f_1})$. Consider the map $\alpha:C_G(x^{f_1})\rightarrow A$ given by $g\mapsto (x,g)$. If $a=(x,g)$ then $x^g=ax$ and therefore $x^{f_1} = (x^{f_1})^g = a^{f_1} x^{f_1}$. Hence the image of $\alpha$ is contained in $\{a\in A : a^{f_1}=1\}$ and this is a subgroup of $A$ of order $\le f_1$. On the other hand $\alpha(g)=\alpha(h)$ if and only if $gh\inv \in C_G(x)$. Therefore
    \begin{equation}\label{CGxfx}
    [C_G(x^{f_1}):C_G(x)] \le f_1.
    \end{equation}

Assume that $K\in \K$ and $y_1$ and $y_2$ are elements of $G$ in the same conjugacy class such that $y_1K=y_2K$. Then $y_2\in y_1A\cap y_1K=\{y_1\}$ because $A\cap K=1$. Therefore, if $C\in \mathcal{C}_{\K}$, then $\{g\in G : (x^f)^g \in y^G K_C\}$ is the disjoint union of the subsets $X_{C,y_1}=\{g\in G : (x^f)^g \in y_1 K_C\}$ with $y_1\in y^G$. If $g,h\in X_{C,y_1}$ then $(x^f)^{gh\inv} =  ((x^f)^hk)^{h\inv} = x^f k^{h\inv}$ for some $k\in K_C$. Then $(x^f,gh\inv)\in A \cap K^{h\inv} = 1$ and hence $gh\inv \in C_G(x^f)$. Conversely, if $gh\inv \in C_G(x^f)$ and $g\in X_{C,y_1}$ then $h\in X_{C,y_1}$. This proves that if $X_{C,y_1}$ is not empty, then it is a coset of $C_G(x^f)=C_G(x^{f_1})$. Therefore for $u_C$ as in Lemma \ref{PartialAugmentationTraces} we get
    \begin{eqnarray}\label{PreDesigualdad3}
    u_C &=& \sum_{y_1 \in y^G} |X_{C,y_1}| \le |C_G(x^{f_1})| \; |y^G| =  |C_G(x^{f_1})|\;  [G:C_G(y)] \\ &\le& f_1 \; |C_G(x)|\;  [G:C_G(y)] , \nonumber
    \end{eqnarray}
%
%
%
%
%
and hence
\begin{equation}\label{Desigualdad3}
\sum_{C\in \mathcal{C}_{\K}}\frac{u_C}{| N_G(K_C)|}\le f_1\frac{| C_G(x)|}{| C_G(y)|}\sum_{C\in \mathcal{C}_{\K}}[G:N_G(K_C)]=
f_1\frac{| C_G(x)|}{| C_G(y)|}|\K|.\end{equation}
%
%
Thus
\begin{equation}\label{Desigualdad4}
\frac{1}{f} [C_G(y):D] \sum_{C\in \mathcal{C}_{\K}} \frac{u_C}{| N_G(K_C)|}\le \frac{1}{f_2}| \K| [C_G(x):D].
\end{equation}

By Lemma \ref{Nucleos}, $|\K| \le |K|$. Moreover, every prime divisor of $m'$ divides $\exp(D)=[D:K]$ and therefore 
    \begin{equation}\label{Desigualdad4.5}
    \frac{1}{f_2} \le \frac{\varphi(f_2)}{f_2}\frac{\varphi(m')}{m'} \frac{[D:K]}{\varphi([D:K])}=\frac{\varphi(m)}{m} \frac{[D:K]}{\varphi([D:K])}.
    \end{equation}
Thus
\begin{equation}\label{Desigualdad5}
\frac{1}{f_2}| \K| [C_G(x):D]\le \frac{\varphi(m)}{m}\frac{[D:K]}{\varphi([D:K])}| K| [C_G(x):D] = \frac{\varphi(m)}{m}\frac{|C_G(x)|}{\varphi([D:K])}.
\end{equation}
Combining (\ref{Dmultiplicidades}), (\ref{Desigualdad4}) and (\ref{Desigualdad5}) we conclude that
    \begin{eqnarray}\label{Desigualdad6}
    \frac{1}{f}\sum_{k\in \K}\mu_{\rho_K(u^f)}(\psi_K(x^f))&\le& \frac{1}{f} [C_G(y):D] \sum_{C\in \mathcal{C}_K} \frac{u_C}{| N_G(K_C)|}\\
    &\le& \frac{\varphi(m)}{m}\frac{| C_G(x)|}{\varphi( [D:K])}. \nonumber
    \end{eqnarray}
This proves (\ref{Desigualdad}).

As the left side of (\ref{KsPartialAugmentation2}) is non-negative we have
    $$\frac{1}{f}\sum_{k\in \K}\mu_{\rho_K(u^f)}(\psi_K(x^f)) \ge -\frac{\varphi(m)}{m}\frac{| C_G(x)|}{\varphi( [D:K])} \varepsilon_x(u) \ge
\frac{\varphi(m)}{m}\frac{| C_G(x)|}{\varphi( [D:K])}$$
because $\varepsilon_x(u)<0$. Therefore the equality holds in (\ref{Desigualdad}) and therefore the equality holds in all the inequalities from (\ref{CGxfx}) to (\ref{Desigualdad6}). This has the following consequences: $\varphi(f_2)=1$, so that $f_2\le 2$; $|C_G(x^{f_1})|=f_1 C_G(x)$ and hence the conjugacy class of $x$ contains all the elements of the form $a^ix$, where $a$ is an element of $A$ of order $f_1$; $m$ is divisible by all the primes dividing $|D|$; $X_{C,y_1}\ne \emptyset$ for every $y_1\in y^G$; the left hand side of (\ref{KsPartialAugmentation2}) is zero and hence $\psi_K(x)$ is not an eigenvalue of $\rho_G(u)$ for every $K\in \K$. Applying this to conjugates of $x$ we deduce that $\psi_K(x^g)$ is not an eigenvalue of $\rho_K(u)$ for every $g\in G$ and every $K\in \K$.
We claim that $f_2\ge 0$ and all the eigenvalues of $\rho_K(u)$ have even order.
If $\xi$ is an eigenvalue of $\rho_K(u)$ then $\xi^f$ is an eigenvalue of both $\rho_K(u^f)$ and $\rho_K(y)$. Thus $\xi^f=\psi_K(y_1)$ for some $y_1\in y^G$. As $X_{C,y_1}\ne\emptyset$, $\xi^f=\psi_K((x^f)^g)$ for some $g\in G$ and therefore $\xi = \zeta_f^j \psi_K(x^g)$ for some $j$. However, $a^i x^g$ is conjugate to $x^g$ for every $0\le i<f_1$, where $a$ is an element of $A$ of order $f_1$. Hence $\psi_K(a^i x^g)=\zeta_{f_1}^i \psi_K(x^g)$ is not an eigenvalue of $\rho_K(u)$ for every $0\le i < f_1$. This implies that $f_2=2$ and $j$ is odd. Therefore the order of $\xi$ is even. This proves the claim. Then all the eigenvalues of $\rho_K(u^{\frac{m}{2}})$ are equal to $-1$ and hence $\psi_K^G(u^{\frac{m}{2}})=-[G:D]$. However $u^{\frac{m}{2}}$ is conjugate to an element of $G\setminus D$ and therefore $\psi_K^G(u^{\frac{m}{2}})=0$, a contradiction.
\end{proof}

\bibliographystyle{amsalpha}
\bibliography{References}

\end{document}